\documentclass[11pt]{amsart}
\input epsf
\usepackage{latexsym}
\usepackage{amssymb,amsmath,amsthm,amscd, amssymb,
graphicx, enumerate, verbatim}
\usepackage[all]{xy}

\numberwithin{equation}{section}
\newtheorem{cor}[equation]{Corollary}

\newtheorem{prop}[equation]{Proposition}
\newtheorem{thm}[equation]{Theorem}

\newtheorem{Example}[equation]{Example}
\newenvironment{ex}{\begin{Example}\rm}{\end{Example}}
\newtheorem{remark}[equation]{Remark}
\newenvironment{rmk}{\begin{remark}\rm}{\end{remark}}

\newcommand{\e}{\varepsilon}
\newcommand{\II}{\mbox{I\!I}}

\def\G{\Gamma}

\def\d{\partial}

\def\S1{\bf S^1}

\newcommand{\ddr}{{\d_r}}

\newcommand{\chn}{{\mathbf{CH}^n}}

\newcommand{\chm}{{\mathbf{CH}^{n-1}}}

\begin{document}

\abovedisplayskip=6pt plus3pt minus3pt
\belowdisplayskip=6pt plus3pt minus3pt

\title[An assortment of negatively curved ends]
{\bf An assortment of negatively curved ends}
\rm

\author{Igor Belegradek}

\address{Igor Belegradek\\School of Mathematics\\ Georgia Institute of
Technology\\ Atlanta, GA 30332-0160}\email{ib@math.gatech.edu}


\date{}
\begin{abstract} 
Motivated by recent groundbreaking work of Ontaneda,
we describe a sizable class of closed manifolds such that the product
of each manifold in the class with $\mathbb R$ admits a complete metric of bounded
negative sectional curvature which is an exponentially warped near
one end and has finite volume near the other end. 
\end{abstract}

\maketitle

\section{Introduction}

Let $V$ be a finite volume, complete, open Riemannian manifold
of sectional curvature $K$ within $[-1,0)$. Little is known
about the topology of $V$.
An argument of Gromov~\cite{Gro-jdg78} extended by 
Schroeder~\cite[Appendix 2]{BGS} implies that $V$ is diffeomorphic 
to the interior of a compact manifold $\overline V$
with boundary; thus any manifold compactification of $V$
is obtained by attaching an h-cobordism to 
$\overline V$ along $\d\overline V$. 

A well-known problem is to determine which manifolds occur as the 
boundary of $\overline V$. Gromov showed that $\d\overline V$ must 
have zero simplicial volume~\cite[p.37]{Gro-vol-bd-coh}, and this seems
to be the only known obstruction in the case when each component
of $\d\overline V$ is aspherical. Other obstructions were found
in~\cite{BelPha} when $\d\overline V$ has a non-aspherical component, 
denoted $C$; e.g. $\pi_1(C)$ cannot be an irreducible,
higher rank lattice, or a virtually nilpotent group.

Earlier examples of manifolds appearing as
components of $\d\bar V$ include 
generalized graph manifolds~\cite{AbrSch, Buy}, and
some circle bundle over real and 
complex hyperbolic manifolds~\cite{Fuj-warp, Bel-rh-warp, Bel-ch-warp}.
A recent breakthrough of Ontaneda allows to dramatically expand the list.

If a (not necessarily connected) manifold $B$ is diffeomorphic to the 
boundary of a connected, smooth, compact manifold $N$, then we say 
that $B$ {\it bounds} $N$, and if $N$ is not specified, we simply say $B$ {\it bounds}.

Ontaneda's proof starts with a closed manifold $B$ that bounds, and
applies relative strict hyperbolization of Charney-Davis~\cite{ChaDav-strict} 
to realize $B$ as a boundary of a compact manifold whose interior
admits a piecewise hyperbolic metric. Then under the assumption
that the building block in the hyperbolization is ``large enough''
Ontaneda is able to smooth the metric away from the boundary
to a Riemannian metric with $K$ near $-1$; this smoothing argument is 
technological tour de force. 
Near the boundary the metric has to be constructed by an ad hoc
method depending on $B$.
The following result is implicit in~\cite{Ont-pinch-hyperb}. 

\begin{thm} 
\label{thm-intro: ont}
{\bf (Ontaneda)}\ Let $B$ be a closed $n$-manifold that bounds and satisfies
the following:
\newline \textup{(i)}
If $n\ge 5$, then any h-cobordism from $B$ to another manifold is a product.\newline
\textup{(ii)}  $\mathbb R\times B$  admits 
a complete metric $g$ of sectional curvature within 
$[-1, 0)$ such  \phantom{\textup{(ii)}}
that $(-\infty, 0\,]\times B$ has finite $g$-volume, and
$g=dr^2+e^{2r} g_{_B}$ on $[c,\infty)\times B$  \phantom{\textup{(ii)}} for some $c>0$ and
a metric $g_B$ on $B$. \newline
Then $B$ bounds a manifold whose interior
admits a complete metric of finite volume and
sectional curvature in $[-1,0)$. 
\end{thm}

Condition (ii) implies that each component of
$B$ is aspherical, and hence has torsion-free fundamental group. 
The Whitehead Torsion Conjecture, which is true for many groups of geometric 
origin, 
predicts that all torsion-free groups have 
zero Whitehead torsion. If the conjecture is true for
the fundamental group of each component of $B$, then (i) holds.

For example (ii) is true if 
$B$ is any infranilmanifold~\cite{BK-GAFA}, or any 
$3$-dimensional $Sol$ manifolds~\cite{Pha-sol}.
We add to the list as follows. 

\begin{thm} 
\label{thm-intro: class B} 
Condition \textup{(ii)} holds for every
manifold in the class $\mathcal B$ that is defined as
the smallest class of closed manifolds 
of positive dimension such that
\newline
$\bullet$ $\mathcal B$ contains each infranilmanifold,
every circle bundle of type  \textup{(K)}, 
and each \phantom{$\bullet$} closed manifold of $K\le 0$ with a local 
Euclidean de Rham factor;\newline
$\bullet$ $\mathcal B$ is closed under products, 
disjoint unions, and products with any compact \phantom{$\bullet$}
manifold of $K\le 0$.
\end{thm}

An orientable circle bundle {\it has type \textup{(K)}} if 
its base is a closed complex hyperbolic $n$-manifold $M$ whose
holonomy representation $\pi_1(M)\to PU(n,1)$ lifts to $U(n, 1)$, and if the Euler class of the bundle
equals $-m\frac{\omega}{4\pi}$ for some nonzero integer $m$, where
$\omega$ is the K\"ahler form of $M$. For example, each nontrivial orientable circle
bundle over a genus two orientable closed surface
has type (K). 

It is immediate from~\cite{BarLue-borelconj, BFL} that
(i) holds for every manifold in $\mathcal B$, so
Theorems~\ref{thm-intro: ont}--\ref{thm-intro: class B} imply

\begin{cor}\label{into-cor class B}
If a manifold $B$ in $\mathcal B$ bounds,
then $B$ bounds a compact connected manifold whose interior $V$
admits a complete Riemannian metric of finite volume and
sectional curvature in $[-1,0)$.
\end{cor}

Theorem~\ref{thm-intro: class B} is proved by showing that 
each manifold in $\mathcal B$ carries what
we call a simultaneously diagonalizable family of metrics $g_r$
such that $g=dr^2+g_r$ is as in (ii), and the main observation
is that simultaneous diagonalizability
behaves well under products, and facilitates curvature computations.

Results of this paper can be modified to hold for
other curvature conditions such as $K\le -1$, or $-1\le K\le 0$,
and other growth assumptions on the ends, such as
infinite volume or $\mathrm{Rad}\,\mathrm{Inj}\to 0$, 
which is all left to an interested reader to explore.

{\bf Acknowledgments} The author is grateful for NSF support (DMS-1105045).

\section{Curvature formulas}

Let us review the curvature 
formulas for the metric $g=dr^2+g_r$ 
on $I\times B$ that were derived in~\cite[Section 6]{BelWei} 
and corrected in~\cite[Appendix C]{Bel-rh-warp},
where $I$ is an open interval, and $B$ is a manifold.
The family of metrics $(B, g_r)$, $r\in I$ is called
{\it simultaneously diagonalizable\,} if near 
each point of $B$ there is a basis of vector 
fields $\{X_i\}$ that is $g_r$-orthogonal
for each $r$.  
Set $h_i(r)=\sqrt{g_r(X_i,X_i)}$ and note that $Y_i=X_i/h_i$
form a $g_r$-orthonormal basis.
Denote $g(X,Y)$, $\frac{\d}{\d r}$ by 
$\langle X , Y\rangle$, $\ddr$, respectively. 
If $g_r$ are simultaneously diagonalizable, the
components of the curvature tensor of $g=dr^2+g_r$ are
\begin{eqnarray}\label{form: curv of warped prod}
& \langle R_g (Y_i,Y_j) Y_j,Y_i\rangle=
\langle R_{g_r}(Y_i,Y_j) Y_j,Y_i\rangle -
\frac{h_i^\prime h_j^\prime}{h_ih_j},\\
& \langle R_g(Y_i,Y_j) Y_l,Y_m\rangle=
\langle R_{g_r}(Y_i,Y_j) Y_l,Y_m\rangle
\ \ \ \mathrm{if}\ \{i,j\}\neq \{l,m\},\\
  & \langle R_g(Y_i,\ddr)\ddr , Y_i \rangle=
-\frac{h_i^{\prime\prime}}{h_i},\ \ \ \ \ \langle
R_g(Y_i,\ddr)\ddr, Y_j \rangle=0\ \ \ \mathrm{if}\ i\neq j\\
&2\langle R_g(\ddr, Y_i) Y_j,Y_k\rangle=
\end{eqnarray}
\vspace{-10pt}
\begin{eqnarray*}
\langle [Y_i,Y_j],Y_k\rangle
\left(\ln\frac{h_k}{h_j}\right)^\prime
+\langle [Y_k,Y_i],Y_j\rangle
\left(\ln\frac{h_j}{h_k}\right)^\prime
+\langle [Y_k,Y_j],Y_i\rangle
\left(\ln\frac{h_i^2}{h_jh_k}\right)^\prime 
\end{eqnarray*}
The second fundamental form $\II_{g_r}$ of $g_r$ is given by
$\II_{g_r}(Y_i, Y_j)=0$ if $i\neq j$, and 
$\II_{g_r}(Y_i, Y_i)=-\frac{h_i^\prime}{h_i}\d_r$~\cite{BelWei}.

\section{Products}
\label{sec: products}

Given manifolds $B_1$, $B_2$
consider the metrics $dr^2+g_{r,k}$ on $\mathbb R\times B_k$ 
with $k=1,2$,
and $g=dr^2+g_r$ on $\mathbb R\times B_1\times B_2$ where
$g_r=g_{r,1}+g_{r,2}$.
Denote the negative reals by $\mathbb R_-$.

\begin{thm} 
For $k=1,2$ suppose that $g_{r,k}$ is simultaneously diagonalizable, and
{\rm $\II_{g_r}$} is negative definite for all $r$. Then\newline
\textup{(1)} $K_{g}<0$ if $K_{dr^2+g_{r,1}}$, $K_{dr^2+g_{r,2}}$ are negative.\newline
\textup{(2)} $K_g$ is bounded if $K_{dr^2+g_{r,1}}$, $K_{dr^2+g_{r,2}}$, 
{\rm $\II_{g_r}$} are bounded.\newline
\textup{(3)} $\mathrm{Vol}(\mathbb R_-\times B_1\times B_2, g)$ is finite if 
$\mathrm{Vol}(\mathbb R_-\times B_1, dr^2+g_{r,k})$ and $\mathrm{Vol}(B_2, g_{0,3-k})$ 
\phantom{(3)} are finite for some $k$.
\end{thm}

\begin{rmk}
If $\II_{g_r}$ is positive definite, the same result holds
with $\mathbb R_-$ replaced by the positive reals. 
\end{rmk}

\begin{proof}
The product $g_r=g_{r,1}+g_{r,2}$ is simultaneously diagonalizable,
so let $Y_i$ be the corresponding basis vectors tangent to one of
the factors.
For brevity denote $g\langle R_g (A,B) C, D\rangle$ by
$\left( A, B, C, D\right)_g$.
Fix vectors $C$, $D$, and write $C=C_1+C_2$, $D=D_1+D_2$
where $C_k$, $D_k$ are tangent to $B_k$. Fix reals $a, b$
such that $a\d_r+C$, $b\d_r+D$ are orthonormal.

Think of $g$ is the Riemannian submersion metric with base $dr^2+g_{r,k}$ and fiber
$g_{r,3-k}$. Since the submersion metric $g$ is a warped product, its $A$-tensor 
vanishes~\cite[9.26]{Bes-book}, so
by O'Neill's formula~\cite[9.28f]{Bes-book} 
$R_g$ restricted to the horizontal
space equals $R_{dr^2+g_{r, k}}$, and hence the $g$-sectional curvature of any horizontal
plane is negative under the assumptions of (1) and bounded under the assumptions of (2). 
Thus if we show that
\[
\mbox{\footnotesize
$\left( a\d_r+C, b\d_r+D, b\d_r+D, a\d_r+C\right)_g\le
\displaystyle{\sum_{k=1}^2}
\left( a\d_r+C_k, b\d_r+D_k, b\d_r+D_k, a\d_r+C_k\right)_g$}
\]
then (1) would follow. Now $\left( a\d_r+C, b\d_r+D, b\d_r+D, a\d_r+C\right)_g$ equals
\[
2ab\left( \d_r, D, \d_r, C\right)_g +
a^2\left( \d_r, D, D, \d_r\right)_g +
2a\left( \d_r, D, D, C\right)_g +
\]
\[
b^2\left( C, \d_r, \d_r, C\right)_g +
2b\left( C, \d_r, D, C\right)_g+
\left( C, D, D, C\right)_g 
\]
and 
each $\left( a\d_r+C_k, b\d_r+D_k, b\d_r+D_k, a\d_r+C_k\right)_g$ has 
a similar decomposition into six summands.
The desired inequality is to be established one summand at a time,
and in fact, it is an equality except for the last summand.

As $\left( \d_r, Y_i, \d_r, Y_j\right)_g=0$ for $i\neq j$
we see that 
$\left(\d_r, D, \d_r, C\right)_g=
\left( \d_r, D_1, \d_r, C_1\right)_g+\left( \d_r, D_2, \d_r, C_2\right)_g$,
and similar equalities hold for  
$\left(\d_r, D, D, \d_r\right)_g$, $\left( C, \d_r, \d_r, C\right)_g$. 

As $\left( \d_r, Y_i, Y_j, Y_k\right)_g=0$ unless $Y_i$, $Y_j$, $Y_k$
are tangent to the same factor, we deduce 
$\left( \d_r, D, D, C\right)_g=\left( \d_r, D_1, D_1, C_1\right)_g+
\left( \d_r, D_2, D_2, C_2\right)_g$; a similar formula holds for 
$\left( C, \d_r, D, C\right)_g$.

The last summand is given by the Gauss formula 
\[
\mbox{\small
$\left( C, D, D, C\right)_g=\left( C, D, D, C\right)_{g_r}+
\langle \II_{g_r}(C,D), \II_{g_r}(C,D)\rangle_g- \langle \II_{g_r}(C,C), \II_{g_r}(D,D)\rangle_g$}
\]
and again a similar 
formula holds for $\left( C_k, D_k, D_k, C_k\right)_g$.

Since $g_r$ is the product metric, 
$\left( C_i, D_j, D_k, C_m\right)_{g_r}=0$ unless $i,j,k, l$
are all equal, so $\left( C, D, D, C\right)_{g_r}=
\left( C_1, D_1, D_1, C_1\right)_{g_r}+\left( C_2, D_2, D_2, C_2\right)_{g_r}$. 
Let $T$ be the diagonal matrix with $ii$ entry equal to 
$\sqrt{\frac{|h_i^\prime|}{h_i}}$. Since $\II_{g_r}$ is nonpositive definite,
$\langle TA, TB\rangle_g \d_r=-\II_{g_r}(A,B)$, so
\[
\mbox{\small
$\langle \II_{g_r}(C,D), \II_{g_r}(C,D)\rangle_g- \langle \II_{g_r}(C,C), \II_{g_r}(D,D)\rangle_g=
\langle TC,TD\rangle_g^2 - \langle TC,TC\rangle_g \langle TD,TD\rangle_g$}.
\] 
The same formula holds for $C_k, D_k$ in place of $C, D$.
Combining the above with vanishing of $\langle TC_i, TC_j\rangle_g$,
$\langle TD_i, TD_j\rangle_g$, $\langle TD_i, TC_j\rangle_g$ for $i\neq j$
gives
\[
\left( C, D, D, C\right)_g-
\left( C_1, D_1, D_1, C_1\right)_g-\left( C_2, D_2, D_2, C_2\right)_g=
\]
\[
\left(\langle TC_1, TD_1\rangle_g+\langle TC_2, TD_2\rangle_g\right)^2-
(|TC_1|^2_g+|TC_2|^2_g)(|TD_1|^2_g+|TD_2|^2_g)\le
\]
\[
\left(|TC_1|_g|TD_1|_g+|TC_2|_g |TD_2|_g\right)^2-
(|TC_1|^2_g+|TC_2|^2_g)(|TD_1|^2_g+|TD_2|^2_g)=
\]
\[
-\left(
|TC_1|_g |TD_2|_g-|TC_2|_g |TD_1|_g
\right)^2\le 0.
\]
which completes the proof of (1). 

To prove (2) note that the above inequalities are
equalities except in one case where the difference
is controlled by $\displaystyle{\max_{i,r}\frac{|h_i^\prime|}{h_i}}$,
which is bounded. Finally $\left( a\d_r+C_k, b\d_r+D_k, b\d_r+D_k, a\d_r+C_k\right)_g$ 
is the product of two bounded quantities, 
$g$-area of the parallelogram and $dr^2+g_{k,r}$-sectional curvature
of the plane each spanned by $\{a\d_r+C_k, b\d_r+D_k\}$,
so (2) follows.

Since $\II_{g_r}$ is nonpositive definite, each $h_i$ is nondecreasing,
so the identity map $(B_k, g_{0,})\to (B_k, g_{r, k})$ with $r<0$
is $1$-Lipschitz, and hence volume nonincreasing. Thus  
if $B_k$ has finite $g_{0,k}$-volume, then
the $g_{r, k}$-volume of $B_k$ is uniformly 
bounded on $\mathbb R_-$. Now (3) follows from
the Fubini theorem for the Riemannian submersion metric $g$
with base $dr^2+g_{r,k}$ and fiber $g_{r,3-k}$.
\end{proof}

\section{Nonpositively curved manifolds}
\label{sec: npc}

Here is a souped up version of the fact that
$dr^2+e^{2r} g_{_B}$ has $K<0$ whenever
$K_{g_{_B}}\le 0$.

\begin{thm} 
\label{thm: bounded npc}
If $(B, g_{_B})$ is a manifold of bounded nonpositive curvature,
then there is a convex, increasing, smooth, positive function $h$ 
that equals $e^r$ for large $r$, and such that the sectional curvature
$(\mathbb R\times B, dr^2+h^2 g_{_B})$ is negative and bounded below. 
\end{thm}
\begin{proof} 
Any $2$-plane tangent to $\mathbb R\times B$ 
is of the form
$\mathrm{span}\{X_1, cX_2+d\d_r\}$ where each $\{X_1, X_2\}$ are $g_{_B}$-orthonormal
vectors tangent to the $B$-factor and $c^2+d^2=1$. Let $Y_i:=X_i/h$, so that $Y_1$, $Y_2$, $\d_r$ is $g$-orthonormal, where $g=dr^2+h^2 g_{_B}$.
Then $\langle R_g(Y_i, Y_j, \d_r) Y_k\rangle_g=0$ and
\[
K_g(Y_1, cY_2+d\d_r)=
c^2\,\left(\frac{K_{g_B}(Y_1, Y_2)}{h^2}-
\left(\frac{h^\prime}{h}\right)^2\right)-
d^2\,\frac{h^{\prime\prime}}{h}.
\]
which is negative if $h^\prime$, $h^{\prime\prime}$ are negative and $K_{g_B}\le 0$.
To make $K_g$ bounded below fix any $\tau>0$, and let
$h$ be an increasing strictly convex function such that $h(r)=e^{r}+\tau$ for $r<0$ and 
$h(r)=e^r$ for $r\ge r_\tau$, which exists is $r_\tau$ is large enough.
(e.g. $h$ can be obtained by smoothing the function
that equals $e^{r}+\tau$ for $r<0$, equals $r+1+\tau$ until the line intersects the graph
of $e^r$, and equals $e^r$ after that).
Then $K_g(Y_1, cY_2+d\d_r)$ is negative and bounded. 
\end{proof}

\begin{cor} Let $(B, g_{_B})$ be a compact
manifold of nonpositive curvature whose universal cover has
a Euclidean de Rham factor. Then $B$ has a simultaneously diagonalizable family
of metrics $g_r$ such that $g_r=e^{2r} g_{_B}$ for large $r$,
$K_{dr^2+g_r}$ is negative and bounded below,
and $\mathrm{Vol} (\mathbb R_-\times B, dr^2+g_r)$ is finite.
\end{cor}
\begin{proof}
Let  $(\mathbb R^s, g_0)\times (X, g_X)$ be a de Rham splitting
of the universal Riemannian cover $(\tilde B, \tilde{g}_{_B})$ of $(B, g_{_B})$,
where $g_0$ is the standard Euclidean metric, $s>0$, 
and $X$ has no Euclidean factors. 
Consider the metric $\tilde g=dr^2+ e^{2r} g_0+ h^2 g_X$ on $\mathbb R\times \tilde B$
where $h$ is as in Theorem~\ref{thm: bounded npc},
and by that theorem $K_{\tilde g}$ is negative and bounded below.
Since the de Rham decomposition
is invariant under isometries, the metric descends to a metric $dr^2+g_r$
on $\mathbb R\times B$ with the same curvature bounds.
By compactness $B$ is covered by finitely many ``product charts'' that
lift isometrically to the universal cover to the product of balls
in $\mathbb R^s$ and $X$. The product of any such chart with $\mathbb R_-$
has finite $g$-volume, and hence so does $\mathbb R_-\times B$.
\end{proof}

\section{Infranilmanifolds}
\label{sec: infra}

Let $B$ be an infranilmanifold of nilpotence degree $k$. 
The metric on $B$ constructed~\cite{BK-GAFA} can be modified to satisfy 
Theorem~\ref{thm-intro: ont}(ii),
while keeping the sectional curvature almost in $[-k^2, -1]$.
This was used but not explained in~\cite{Ont-pinch-hyperb},
so we supply a proof below. Here ``almost'' means that for any $\e>0$
there is a metric as in (ii) with sectional curvature within $[-(k+\e)^2, -1]$.

The metric in~\cite{BK-GAFA} is of the form $g=dr^2+g_r$, where
$g_r$ is simultaneously diagonalizable,
and the only part of (ii) that does not hold for $g$ is that
for all $r\ge c$ we have $g_r=e^{2kr}g_B$ where $g_B$ is a fixed metric. 
To satisfy (ii) we need to ``replace'' $k$ by $1$. 

To do so choose a non-increasing 
function $Q$ of $r$ such that $Q(r)=k$ on $[c, T_1]$ and $Q(r)=1$ for $r\ge T_2$.
By making $T_2\gg T_1$ we may assume that $Q^\prime$ is uniformly small.
Set $q:=kc+\int_c^{r} Q(r) dr$ so that $q(r)=kr$ on $[c,T_1]$ and $q(r)=r+(k-1)c$
for $r\ge T_2$. Set $h:=e^q$ and consider the metric $\bar g:=dr^2+\bar g_r$
such that $\bar g_r=g_r$ for $r\le T_1$, and $\bar g_r=h^2 g_B$ for $r\ge c$;
the two definitions of $\bar g_r$ agree on the overlap $[c, T_1]$.
For $r\ge T_2$ we have $\bar g_r=e^{2r} e^{2c(k-1)} g_B$, i.e. $\bar g_r$
is the $e^{2r}$ multiple of a fixed metric, as desired. 

It remains to show
that the sectional curvature $\bar g$ is almost in $[-k^2, -1]$ for suitable $T_1$, 
$T_2$, and for this was essentially done in~\cite{BK-GAFA}. Indeed, if $r\le T_1$,
then $\bar g_r=g_r$, so~\cite{BK-GAFA} applies directly, and if $r\ge T_1$, 
then the proof in~\cite{BK-GAFA} shows that up to a small error,
the sectional curvature of $\bar g$ is expressed via the quantities 
$\left(\frac{h^\prime}{h}\right)^2=Q^2$ and
$\frac{h^{\prime\prime}}{h}=Q^\prime+Q^2$, and the desired claim
follows since $Q^\prime$ is small, and $Q\in [1, k]$.

\section{Circle bundles of type (K)}
\label{sec: type K}

In this section we describe some circle bundles whose total spaces
lie in $\mathcal B$.

\begin{prop} 
The total space of an oriented circle bundle over 
closed nonpositively curved manifold
has a nonpositively curved metric
if and only if the bundle has torsion Euler class.
\end{prop}
\begin{proof} 
Consider an oriented (or equivalently principal)
circle bundle with total space $E$ and base $M$.

For the ``only if'' direction recall that in CAT($0$) groups
centralizers virtually split~\cite[Theorem II.7.1(5)]{BH-book}, 
so a fiber lifts homeomorphically to a finite cover $\hat E$ of $E$ where it represents
the $\mathbb Z$-factor in the splitting 
$\pi_1(\hat E)\cong\mathbb Z\times H$ for some group $H$. 
It follows that the free circle
action associated with the bundle lifts to a free circle action 
on $\hat E$, see~\cite[Theorem I.9.1]{Bre-book}. Thus $\hat E$ is the total space
of an oriented circle bundle, and the covering $\hat E\to E$
descends to a map of orbit spaces $\hat M\to M$, so that the bundle structure
on $\hat E$ is the pullback of the bundle structure on $E$. 
The bundle projection $\hat E\to\hat M$ maps $H$ isomorphically
onto $\pi_1(\hat M)$, and since $\hat E$ is aspherical,
the bundle has a homotopy section, and hence its Euler class vanishes. 
Therefore, the rational Euler class of $E\to M$ vanishes because finite 
covers are injective on rational cohomology by a transfer argument.

For the ``if'' direction, note that vanishing of the rational Euler class
forces the bundle to be flat Euclidean, see~\cite{Miy-tors-eul-cl} 
or~\cite{OprTan-flat-bdl}. Thus its total space $E$ is a $\pi_1(M)$-quotient 
of $\widetilde M\times S^1$,
where $\pi_1(M)$ acts by deck-transformations on the universal cover 
$\widetilde M$ and via a homomorphism
$\pi_1(M)\to S^1$ on the second factor. The action is isometric, so 
the total space carries a metric of $K\le 0$.
\end{proof}

\begin{rmk}
By the above proof, if the bundle has torsion Euler class,
the metric of $K\le 0$ on the total space has a local Euclidean 
de Rham factor, and hence the total space lies in $\mathcal B$ 
by Section~\ref{sec: npc}.
\end{rmk}

Deciding which circle bundles with non-torsion Euler class
lie in $\mathcal B$ seems much harder; to date I can only handle 
type (K) circle bundles, and the proof given below depends on one of the main results 
of~\cite{Bel-ch-warp}.

We seek to construct a manifold $\mathbb R\times B$ as in Theorem~\ref{thm-intro: ont}(ii)
such that $B$ is a circle bundle with non-torsion Euler class
over a closed complex hyperbolic $(n-1)$-manifold $M$.
Thus $M$ is the quotient of $\chm$ by a cocompact torsion-free lattice in $PU(n-1,1)$,
where $\mathbf{CH}^l$ denotes the complex hyperbolic space of complex dimension $l$
normalized to have holomorphic sectional curvature $-1$.
The strategy is to realize $M$ as a totally geodesic submanifold
in a complete complex hyperbolic $n$-manifold $V$ that is homotopy equivalent to $M$, 
and then modify the (incomplete) metric on $V\smallsetminus M$ to make it 
as in Theorem~\ref{thm-intro: ont}(ii).
Since there is no inclusion of $PU(n-1, 1)$ into $PU(n,1)$, in order
to produce $V$  with desired properties we assume that $M$ has type (K), i.e.
its holonomy representation $\pi_1(M)\to PU(n-1,1)$ lifts to $U(n-1,1)$,
so that we can compose the lift with the inclusion into $U(n,1)$
followed by the projectivization to $PU(n,1)$.

To this end fix a torsion-free cocompact lattice in $U(n-1,1)$,
where $n\ge 2$, and let $\hat\G$ be the image of the lattice
under the inclusion $U(n-1,1)\hookrightarrow U(n,1)$.
Let $\G$ denote the projecton of $\hat\G$ in $PU(n,1)$, which acts 
on $\chn$ holomorphic isometries. The kernel
of the projectivization $U(n,1)\to PU(n,1)$ consists of scalar matrices
that form a diagonally embedded $U(1)$, so since $\hat\G$ is discrete and torsion-free,
the projection $\hat\G\to\G$ is an isomorphism.

Then $\G$ acts freely on $\chn$, and it
stabilizes a subspace $\chm\subset\chm$ on which it acts with compact quotient 
$M:=\chm/\G$; thus $M$ is a compact, totally geodesic, embedded submanifold of 
$\bar M:=\chn/\G$.
The metric on $\chn$ can be written in cylindrical coordinates about $\chm$
as $dr^2+ g_r$ for $g_r=\sinh^2(r) d\phi^2 + \cosh^2\left(\frac{r}{2}\right) \mathbf{k}^{n-1}$ 
where $r\ge 0$ and $\mathbf{k}^{n-1}$ 
is the standard metric on $\chm$, see~\cite[Section 3]{Bel-ch-warp}. 
Let $F$ denote the unit normal bundle to $\chm$ in $\chn$, and 
set $\bar F:=F/\G$.
The proof of~\cite[Theorem 1.1(iii)]{Bel-ch-warp} yields a complete metric 
on $\mathbb R\times F$ of the form $dr^2+v^2 d\phi^2 + h^2 \mathbf{k}^{n-1}$.
The metric has sectional curvature in $[a,0)$ for some negative constant $a$, 
and given $\e>0$ can be chosen to satisfy 
$v=\sinh(r)$ and $h=\cosh\left(\frac{r}{2}\right)$ for $r\ge \e$. 
The metric is invariant under the stabilizer of $\chm$
in the isometry group of $\chn$, and hence it descends to a complete metric
on $\mathbb R\times\bar F$, which we denote $g_{v,h}$. By construction in~\cite{Bel-ch-warp}
the portion $(-\infty, 0\,]\times\bar F$ has finite $g_{v,h}$-volume.

The subgroup $U(n-1,1)$ of $U(n, 1)$ commutes with the copy of $U(1)$ 
where $U(n-1,1)\times U(1)$ is embedded into $U(n,1)$ via the map
\[
(A,B)\to  \left[
\begin{array}{rr}
A & 0\\
0 & B 
\end{array}\right],
\] and hence their images in $PU(n,1)$ also commute. 
It follows that there is an isometric $U(1)$-action on $\chn$ 
by rotation about $\chm$, and the corresponding action 
on $(\mathbb R\times \bar F, g_{v,h})$ is also isometric and free, where
$U(1)$-orbits are the fibers of the normal circle bundles 
$\{r\}\times\bar F\to M$.

Let $\bar F_m$ denote the quotient of $\bar F$ by $\mathbb Z_m\le U(1)$; 
this is a principal circle bundle over $M$ whose Euler class is 
the $m$th multiple of the Euler class of $\bar F$. 
Let $\bar g_{v,h,m}$ denote the metric on $\mathbb R\times \bar F_m$
that is descended from $(\mathbb R\times \bar F, g_{v,h})$; of course,
$\bar g_{v,h,m}$ and $g_{v,h}$ are locally isometric.

As is explained e.g. in~\cite[Lemma 13.1]{Bel-ch-warp} the first Chern class
of normal bundle of $M$ in $\chn/\G$ is represented $-\frac{\omega}{4\pi}$ 
where $\omega$ is the K\"ahler form of the complex hyperbolic metric on $\chm/\G$.
Thus the Euler class of the circle bundle $\bar F_m\to M$ equals $-m\frac{\omega}{4\pi}$;
we refer to such circle bundles as the type (K). 

\begin{ex}
If $n=1$, then $M$ is a closed orientable surface of negative
Euler characteristic, and according to~\cite[Corollary 2.3.4]{GKL}
$-\frac{\omega}{4\pi}$ integrated over $M$ equals 
$\pm\chi(M)/2$; thus if $M$ has genus $2$, then every
orientable circle bundle over $M$ has type (K).
\end{ex}

The metric $(\mathbb R\times \bar F_m, g_{v,h,m})$ is simultaneously diagonalizable
on the fiber, and it satisfies all
properties needed in (ii) except that it is not equal to $d^2+ e^{2r} g_{_{B_m}}$
for large $r$. In what follows we modify $g_{v,h,m}$ for large $r$ to 
make (ii) hold. The idea is similar to that of Section~\ref{sec: infra}
but details are more involved. Locally
\[
\bar g_{v,h,m}=dr^2+ v^2 d\phi^2 + h^2 \mathbf{k}^{n-1}
\] 
and for $r\ge \e$ we have 
$v=\sinh(r)$ and $h=\cosh\left(\frac{r}{2}\right)$, so
the metric is complex hyperbolic, and its sectional curvature is within $[-1, -\frac{1}{4}]$.
Small $C^2$ change of the warping functions affect the curvature only slightly,
so as a first step we change $v$, $h$ so that 
$v=\frac{e^r}{2}$ and $h=\frac{1}{2}e^\frac{r}{2}$
for $r\ge T_0$ provided $T_0$ is large enough. 
Next we wish to change $h$ to $\frac{e^r}{2}$ for large $r$ while keeping
curvature negative. To this end let $Q$ be a smooth non-decreasing function
such that $Q=\frac{1}{2}$ on $[T_0, T_1]$ and $Q=1$ for $r\ge T_2$. Let 
$q(r):=\frac{T_0}{2}+\int_{T_0}^r Q(r) dr$. Setting $h:=\frac{e^q}{2}$ defines a metric,
which we denote $g_{v,h,m}$, that agrees the previously defined metric for $r\le T_0$, while
if $r\ge T_2$, then
$g_{v,h,m}-dr^2$ is the $e^{2r}$ multiple of the fixed metric as desired for (ii).

It remains to show that 
the curvature remains negatively pinched for $r\ge T_0$. 
Set $s:= \frac{v}{h^2}$; as $h\ge \frac{1}{2} e^{\frac{r}{2}}$ we get $0<s\le 2$
for $r\ge T_0$.
The sectional curvature of any metric of the form
$g_{v,h,m}$ was computed in~\cite[Section 9]{Bel-ch-warp}
in terms of $h$ and $v$, and below we review the results of this computation. 
Let $\{C,D\}$ denote
an orthonormal basis in a two-plane with
\[
C=c_0\ddr+c_1Y_1+c_2Y_2+c_3Y_3, \quad D=d_1Y_1+d_2Y_2,
\]
where $c_i, d_j\in\mathbb R$ and $\{\d_r, Y_1, Y_2, Y_3\}$ are orthonormal, 
and $Y_1$ is tangent to the circle fiber.
The sectional curvature of the plane is then given by
\begin{eqnarray}
\label{form: k(c,d)}
& (d_1c_2-d_2c_1)^2 K (Y_2,Y_1)+
d_1^2c_3^2 K(Y_3,Y_1)+d_1^2c_0^2 K(\ddr,Y_1)+\\
\nonumber &  
d_2^2c_0^2 K(\ddr,Y_2)+
d_2^2c_3^2 K(Y_3,Y_2)+
3d_1d_2c_0c_3 \langle R(\ddr , Y_1) Y_2, Y_3\rangle.
\end{eqnarray}
where 
\begin{eqnarray*} 
& \label{form: k(y_i,y_1)}
K (Y_2,Y_1)=K(Y_3,Y_1)=
\frac{s^2}{16}-\frac{v^\prime}{v}\frac{h^\prime}{h}\le \frac{1}{4}-Q\le -\frac{1}{4},\\
& \label{form: k(y_3,y_2)}
K (Y_3,Y_2)=-\frac{1}{4h^2} - \frac{3}{h^2}c_{23}^2- 
3c_{23}^2\frac{s^2}{4}-\left(\frac{h^\prime}{h}\right)^2 <
-3c_{23}^2\frac{s^2}{4},\\
& \label{form: k(dr,y_i)}
K(\ddr,Y_1)=-\frac{v^{\prime\prime}}{v}=-1, \ \ \ \ \
K(\ddr,Y_2)=-\frac{h^{\prime\prime}}{h}=-Q^\prime-Q^2,\\
& \label{form: r(dr,y_1, y_2, y_3)}
\langle R(\ddr , Y_1) Y_2, Y_3\rangle=
-c_{23}\frac{v}{h^2}\left(\frac{v^\prime}{v}-
\frac{h^\prime}{h}\right)=-c_{23} s (1-Q).
\end{eqnarray*}
where $c_{23}$ is a constant with $|c_{23}|\le\frac{1}{2}$
defined by $[Y_2, Y_3]=c_{23}\frac{v}{h^2} Y_1$.
Now 
\[
d_1^2c_0^2 K(\ddr,Y_1)+d_2^2c_3^2 K(Y_3,Y_2)< -|d_1c_0|^2-|d_2c_3|^2 \frac{3c_{23}^2 s^2}{4}=
\]
\[
-\left(
|d_1c_0|-|d_2c_3|\frac{\sqrt{3}|c_{23}| s}{2}\right)^2-
|d_1c_0d_2c_3\,c_{23}|\sqrt{3}s,
\]
while $0\le 1-Q\le \frac{1}{2}$ implies
\[
3d_1d_2c_0c_3 \langle R(\ddr , Y_1) Y_2, Y_3\rangle\le
|d_1d_2c_0c_3\,c_{23}| \frac{3}{2} s
\]
so these three terms of (\ref{form: k(c,d)}) add up to a nonpositive number.
Since the other three terms are also nonpositive, $K(C,D)\le 0$, and
if $K(C,D)$ is zero somewhere, then the coefficients for the latter three 
terms vanish,
and in particular, $d_1c_3=0$, so the mixed term 
$3d_1d_2c_0c_3 \langle R(\ddr , Y_1) Y_2, Y_3\rangle$ is not present in the sum,
and elementary linear algebra (as in~\cite[Remark 9.6]{Bel-ch-warp}) leads
to a contradiction with $K(C,D)=0$. Thus 
$K(C,D)<0$ for $r\ge T_0$, which
completes the proof that every type (K) circle 
bundle lies in $\mathcal B$.

\begin{rmk}
A similar construction in the real hyperbolic case yields
circle bundles with torsion Euler class, see~\cite{Bel-rh-warp}.
\end{rmk}

\section{Proof of Theorem~\ref{thm-intro: class B}} 
\label{sec: proof of main thm}

That $\mathcal B$ contains each infranilmanifold,
every circle bundle of type  \textup{(K)}, 
and each closed manifold of $K\le 0$ with a local 
Euclidean de Rham factor is proved in 
Sections~\ref{sec: npc}--\ref{sec: type K}
and in each case the metric on the the fiber $B$
is simultaneously diagonalizable. Hence Section~\ref{sec: products} implies
that $\mathcal B$ is closed under products, and obviously
it is closed under disjoint unions.
Finally, if $B$ is a manifold in $\mathcal B$, then inductively it comes with
a family $g_r$ of simultaneously diagonalizable metrics such that $dr^2+g_r$ is as in (ii).
If $(M, g_{_M})$ is a closed manifold of $K\le 0$, and $h$ is as in Theorem~\ref{thm: bounded npc},
then $g_r+h^2g_{_M}$ is simultaneously diagonalizable, and
$dr^2+g_r+h^2g_{_M}$ is as in (ii).

\small

\def\cprime{$'$}
\providecommand{\bysame}{\leavevmode\hbox to3em{\hrulefill}\thinspace}
\providecommand{\MR}{\relax\ifhmode\unskip\space\fi MR }
\providecommand{\MRhref}[2]{%
  \href{http://www.ams.org/mathscinet-getitem?mr=#1}{#2}
}
\providecommand{\href}[2]{#2}

\end{document}